\documentclass[12pt]{article}
\usepackage{graphicx}
\usepackage{amsmath,amsthm,amssymb,amsfonts,euscript,enumerate}
\usepackage{amsthm}
\usepackage{color,wrapfig}
\usepackage{pxfonts}
\usepackage{verbatim}
\usepackage{colortbl}
\newtheorem{theorem}{Theorem}[section]

\newtheorem{lemma}[theorem]{Lemma}
\newtheorem{conjecture}[theorem]{Conjecture}
\newtheorem{problem}[theorem]{Problem}

\newtheorem{defn}{Definition}[section]
\newtheorem{remark}[theorem]{Remark}

\begin{document}
\title{A recurring pattern in natural numbers of a certain property}
\author{Daniel Tsai\\
Graduate School of Mathematics, Nagoya University\\
Furocho, Chikusa-ku, Nagoya 464-8602}
\date{July 19, 2019}
\smallbreak \maketitle
\begin{abstract}
Natural numbers satisfying an unusual property are mentioned by the author in \cite{T}, in which their infinitude is also proved. In this paper, we start with an arbitrary natural number which is not a multiple of $10$ and non-palindromic, form numbers by concatenating its decimal digits, and investigate which of them have the unusual property. In particular, the pattern of which of them have the unusual property recurs.
\end{abstract}

\vskip 0.2truein



\vskip 0.2truein
\setcounter{equation}{0}
\setcounter{section}{0}

\section{Introduction}
\setcounter{equation}{0}
\setcounter{theorem}{0}

An unusual property which some natural numbers, e.g. 198, satisfy are defined by the author in \cite{T}. We see that
\begin{equation*}
\begin{aligned}
198=&2\cdot3^2\cdot11,\\
891=&3^4\cdot11,
\end{aligned}
\end{equation*}
and
\begin{equation*}
2+(3+2)+11=(3+4)+11.
\end{equation*}
That is, the sum of the numbers appearing in the prime factorizations of the two numbers are equal. Notice that the exponents $1$ "does not appear". In general, the definition is, that a natural number $n$ has this property if $10\nmid n$, $n$ is non-palindromic, and that the sum of the numbers appearing in the prime factorization of $n$ is equal to that of the number formed by reversing its decimal digits. In \cite{T}, the infinitude of such numbers is proved, in particular
\begin{equation}\label{18}
\begin{aligned}
18,1818,181818,\ldots&,\\
18,198,1998,19998,\ldots&
\end{aligned}
\end{equation}
all have this property. The first is the sequence of concatenations of $18$; the second is the sequence of numbers $19\ldots 98$, with any number of $9$'s in between. In this paper, we start with an arbitrary non-palindromic natural number $10\nmid n$, form, like in the first sequence in \eqref{18}, numbers by concatenating its decimal digits, and show that there is a recurring pattern in which of them have this property. More precisely, whether one of them have this property depends only on the number of times the digits of $n$ are concatenated to form it modulo some natural number.

  This unusual property, called $v$-palindromic in this paper, is defined using two concepts, namely 
\begin{itemize}
\item Reversing the decimal digits of a natural number $n$. In this paper we only allow $n$ to not be a multiple of $10$, and denote the resulting number by $r(n)$. The reason is that we do not want to have leading digits of $0$ after reversing.
\item In the prime factorization of a natural number
\begin{equation}\label{nintro}
n=p^{e_1}_1p^{e_2}_2\cdots p^{e_m}_m,
\end{equation}
summing all the numbers that appear, i.e. the prime factors and the exponents, but not including an exponent when it is $1$, because they are usually not written. In this paper we denote this sum by $v(n)$, i.e.
\begin{equation}\label{vn}
  v(n) = \sum^m_{i=1}(p_i+\iota(e_i)),
\end{equation}
where $\iota(e)=0$ if $e=1$ and $\iota(e)=e$ if $e\ge2$.

\end{itemize}

About reversing the decimal digits of a natural number, some investigations have been done by others. In \cite{K}, numbers $n$ such that $n$ divides $r(n)$, i.e. $n\mid r(n)$, are mentioned. In particular, all of the numbers in
\begin{equation}\label{2178}
  2178,21978,219978,2199978,\ldots,
\end{equation}
i.e. the sequence of numbers $219\ldots 978$, with any number of $9$'s in between, satisfy $4n=r(n)$. The resemblance of the second sequence in \eqref{18} with \eqref{2178} is a bit interesting. While the relation $n\mid r(n)$ is studied in \cite{K}, the relation $v(n)=v(r(n))$ is studied in this paper. In \cite{G}, non-palindromic prime numbers $p$ such that $r(p)$ is also prime are mentioned, they are called emirps.

About $v(n)$, similar arithmetic functions have been studied. In \cite{A}, assuming \eqref{nintro}, the arithmetic function
\begin{equation}\label{An}
  A(n) = \sum^m_{i=1}p_ie_i
\end{equation}
is studied. Also, the entries A008474 and A000026 of the OEIS \cite{S} are similar to $v(n)$. The entry A008474 is, assuming \eqref{nintro},
\begin{equation}\label{v'}
  v'(n) = \sum^m_{i=1}(p_i+e_i),
\end{equation}
which is almost the same as $v(n)$ except that it has $e_i$ instead of $\iota(e_i)$, i.e. when summing all the numbers that appear in \eqref{nintro}, also including an exponent when it is $1$.

Palindromes are numbers $n$ such that $n=r(n)$. These obviously satisfy $n\mid r(n)$ and also $v(n)=v(r(n))$. Therefore the problem studied in \cite{K}, as well as the content of this paper, are more about non-palindromes, rather than palindromes.

\section{Definition of the unusual property}
\setcounter{equation}{0}
\setcounter{theorem}{0} 

In this section we will recall the definition in \cite{T} of the unusual property. In the following, 
\begin{equation*}
\begin{aligned}
\mathbb{N}_{\ne10}=&\{n\in\mathbb{N}: 10\nmid n\},\\
\mathbb{Z}_{\ge0}=&\{z\in\mathbb{Z}: z\ge0\}.
\end{aligned}
\end{equation*}

\begin{defn}
For $n\in\mathbb{N}_{\ne10}$ with decimal representation $n=d_{k-1}\ldots d_1d_0$, we put
\begin{equation*}
r(n)=d_0d_1\ldots d_k.
\end{equation*}
That is, $r(n)$ is the number formed by writing the decimal digits of $n$ in reverse order. Hence we have $r:\mathbb{N}_{\ne10}\to\mathbb{N}_{\ne10}$. We define $n$ to be palindromic if $n=r(n)$.
\end{defn}

\begin{defn}
We put
\begin{itemize}
\item $v(p)=p$ for $p$ a prime,
\item $v(p^e)=p+e$ for $p$ a prime and $e\ge2$,
\end{itemize}
and insist that $v:\mathbb{N}\to \mathbb{Z}_{\ge0}$ be an additive arithmetic function.
If we put
\begin{eqnarray*}
\iota(e)=\left\{ \begin{array}{ll}
0 & (e=1) \\
e & (e\ge2), \\
\end{array} \right.
\end{eqnarray*}
then we may combine the above two points and just put
\begin{equation*}
v(p^e)=p+\iota(e).
\end{equation*}
\end{defn}

Let $n\ge2$ be a natural number with prime factorization
\begin{equation}\label{n}
n=p^{e_1}_1p^{e_2}_2\cdots p^{e_m}_m.
\end{equation}
Then
\begin{equation*}
v(n)=\sum^m_{i=1}v(p^{e_i}_i)=\sum^m_{i=1}(p_i+\iota(e_i)).
\end{equation*}
Hence $v(n)$ is the sum of the numbers appearing in the prime factorization of $n$, not counting exponents which are $1$.

We may now define the unusual property, which we call $v$-palindromic:

\begin{defn}\label{defnvpal}
A natural number $n$ is $v$-palindromic if $n\in\mathbb{N}_{\neq10}$, $n\neq r(n)$, and $v(n)=v(r(n))$.
\end{defn}

It is clear that if $n$ is $v$-palindromic then so is $r(n)$. As noted in the Introduction, $198$ and the numbers \eqref{18} are $v$-palindromic numbers. In the next section we shall state our main theorem.

\section{Statement of the main theorem}
\setcounter{equation}{0}
\setcounter{theorem}{0}

In this section we shall define some notations to state our main theorem.

\begin{defn}
For $c,k\ge1$, put
\begin{equation}
\rho_{c,k}=\overbrace{
1\underbrace{0\ldots0}_\text{$k-1$}1
\underbrace{0\ldots0}_\text{$k-1$}1
\ldots
1\underbrace{0\ldots0}_\text{$k-1$}1
}^\text{$c$},
\end{equation}
meaning that $1$ appears $c$ times and that between each consecutive pair of them $0$ appears $k-1$ times.
\end{defn}

It is clear that if $n$ is a $k$-digit number then the number formed by repeating its digits $c$ times is just $n\rho_{c,k}$. We may now state our main theorem:

\begin{theorem}\label{main}
Let $n$ be a natural number with $k$ digits and with $n\in\mathbb{N}_{\ne10}$ and $n\ne r(n)$. Then there exists a natural number $\omega>0$ such that for every $c\ge1$, $n\rho_{c,k}$ is $v$-palindromic if and only if $n\rho_{c+\omega,k}$ is. In other words, whether $n\rho_{c,k}$ is $v$-palindromic depends only on $c$ modulo $\omega$.
\end{theorem}

\begin{remark}
In fact the main theorem also holds if in defining $v$-palindromic numbers we used the $v'$ in \eqref{v'} instead of $v$. Moreover the proof will be slightly shorter because one does not have to deal with the subtlety caused by not summing an exponent when it is $1$.
\end{remark}

We make the following definition based on the truth of the above theorem:

\begin{defn}\label{defnofwc}
A natural number $\omega>0$ satisfying the condition of the above theorem is called a period of $n$ and the smallest one is denoted $\omega(n)$. If there exists a $c\ge1$ such that $n\rho_{c,k}$ is $v$-palindromic, the smallest one is called the order of $n$ and denoted $c(n)$. If such a $c$ does not exist then we write $c(n)=\infty$.
\end{defn}

We have the following:

\begin{theorem}\label{period}
The set of all periods of $n$ is $\{q\omega(n):q\in\mathbb{N}\}$.
\end{theorem}

We prove our main theorem in Section \ref{proof}. Before that, we need some preparation. In Section \ref{divisible} we investigate some divisibility properties of the numbers $\rho_{c,k}$. In Section \ref{819} we first consider the case $n=819$ of the main theorem; the proof of the main theorem is essentially a generalization of this.

\section{Divisibility properties of $\rho_{c,k}$}\label{divisible}
\setcounter{equation}{0}
\setcounter{theorem}{0}

We consider the divisibility of the numbers $\rho_{c,k}$ by prime powers $p^\alpha$. Recall that $\operatorname{ord}_p(n)$ is the largest integer $\beta$ with $p^\beta\mid n$. We have:

\begin{lemma}\label{divisiblelemma}
Let $p^\alpha$ be a prime power, with $p\ne2,5$. Let $k\ge1$, let $\beta=\operatorname{ord}_p(10^k-1)$, and let $h$ be the order of $10^k$ regarded as an element of $(\mathbb{Z}/p^{\alpha+\beta}\mathbb{Z})^\times$. Then $h>1$ and for $c\ge1$, $p^\alpha\mid\rho_{c,k}$ if and only if $h\mid c$.
\end{lemma}

\begin{proof}
We first show that $h>1$. That $h=1$ means that
\begin{equation*}
10^k\equiv1\pmod{p^{\alpha+\beta}}
\iff p^{\alpha+\beta}\mid 10^k-1\iff p^{\alpha+\operatorname{ord}_p(10^k-1)}\mid 10^k-1,
\end{equation*}
which cannot be. Whence $h>1$. We have
\begin{equation*}
(10^k-1)\rho_{c,k}=(10^k-1)\sum^{c-1}_{i=0}10^{ki}=10^{kc}-1.
\end{equation*}
As $\beta=\operatorname{ord}_p(10^k-1)$,
\begin{equation*}
p^\alpha\mid\rho_{c,k}\iff 10^{kc}-1\equiv0\pmod{p^{\alpha+\beta}}\iff10^{kc}\equiv1\pmod{p^{\alpha+\beta}}\iff h\mid c,
\end{equation*}
where the last $\iff$ is due to the structure of cyclic subgroups.
\end{proof}

\begin{remark}\label{however}
In Lemma \ref{divisiblelemma}, if $p=2,5$, then $10^k$ cannot be regarded as an element of $(\mathbb{Z}/p^{\alpha+\beta}\mathbb{Z})^\times$. But obviously for every $c\ge1$, $p^\alpha\nmid\rho_{c,k}$. Also, let us denote the $h$ in the lemma by $h_{p^\alpha,k}$.
\end{remark}

Using Mathematica \cite{Math} we have:

  \bigskip
  \begin{center}
    \def\arraystretch{2.2}
  \begin{tabular}{c|c|c|c|c|c|c}
  $p^\alpha$ & $7$ & $7^2$ & $13$ & $13^2$ & $17$ & $17^2$ \\
  \hline \hline
  $h_{p^\alpha,3}$ & $2$ & $14$ & $2$ & $26$ & $16$ & $272$ \\
\end{tabular}
\end{center}
\bigskip
Regarding divisibility in general, not just for $\rho_{c,k}$, we recall that:

\begin{lemma}\label{secondtrans}
Let $n$ be a natural number, let $p$ be a prime, and let $g=\operatorname{ord}_p(n)$.
\begin{enumerate}
\item $g=0 \iff p\nmid n$,
\item $g=1 \iff p\mid n\text{ and }p^2\nmid n$,
\item $g\le1\iff p^2\nmid n$,
\item $g\ge1\iff p\mid n$,
\item $g\ge2\iff p^2\mid n$.
\end{enumerate}
\end{lemma}

We will need this lemma later.

\section{The functions $\varphi_{p,\delta}$}\label{functions}
\setcounter{equation}{0}
\setcounter{theorem}{0}

Fix a prime $p$, the sequence of powers of $p$ is
\begin{equation*}
1,p,p^2,\ldots,p^\alpha,\ldots.
\end{equation*}
Applying $v$ to them yields
\begin{equation*}
0,p,p+2,\ldots,p+\alpha,\ldots.
\end{equation*}
Now we take differences of consecutive terms to get
\begin{equation}\label{difference}
p,2,1,\ldots,1,\ldots,
\end{equation}
with all $1$'s from the third term onwards. We give notation for terms of this sequence

\begin{defn}
For a prime $p$ and integer $\alpha\ge0$, put
\begin{equation*}
\varphi_{p,1}(\alpha)=v(p^{\alpha+1})-v(p^\alpha).
\end{equation*}
\end{defn}

In this notation then, the sequence \eqref{difference} is $(\varphi_{p,1}(\alpha))^\infty_{\alpha=0}$. More generally we define:

\begin{defn}
For a prime $p$, an integer $\alpha\ge0$, and a $\delta\ge1$, put
\begin{equation*}
\varphi_{p,\delta}(\alpha)=v(p^{\alpha+\delta})-v(p^\alpha).
\end{equation*}
\end{defn}

In this notation, for instance, the sequence $(\varphi_{p,3}(\alpha))^\infty_{\alpha=0}$ is
\begin{equation*}
p+3,4,3,\ldots,3,\ldots,
\end{equation*}
with all $3$'s from the third term onwards. More generally, for $\delta\ge2$, the sequence $(\varphi_{p,\delta}(\alpha))^\infty_{\alpha=0}$ is just
\begin{equation}\label{differencedelta}
p+\delta,\delta+1,\delta,\ldots,\delta,\ldots,
\end{equation}

We may view, for a prime $p$ and $\delta\ge1$, $\varphi_{p,\delta}:\mathbb{Z}_{\ge0}\to\mathbb{N}$ as a function of $\alpha\in\mathbb{Z}_{\ge0}$.

Rephrasing \eqref{difference} and \eqref{differencedelta}, the values of $\varphi_{p,\delta}$ may be summarized as follows.
\begin{eqnarray}\label{varphi21}
\varphi_{2,1}(\alpha)=\left\{ \begin{array}{ll}
2 & (\alpha=0,1) \\
1 & (\alpha\ge2), \\
\end{array} \right.
\end{eqnarray}
\begin{eqnarray}\label{varphip1}
\varphi_{p,1}(\alpha)=\left\{ \begin{array}{ll}
p & (\alpha=0) \\
2 & (\alpha=1) \\
1 & (\alpha\ge2) \\
\end{array} \right.\quad\text{if } p\ne2,
\end{eqnarray}
and
\begin{eqnarray}\label{varphipdelta}
\varphi_{p,\delta}(\alpha)=\left\{ \begin{array}{ll}
p+\delta & (\alpha=0) \\
\delta+1 & (\alpha=1) \\
\delta & (\alpha\ge2) \\
\end{array} \right.\quad\text{if } \delta\ge2.
\end{eqnarray}
Where we have deliberately distinguished between the cases where the values are distinct. We give a notation for the ranges of $\varphi_{p,\delta}$:

\begin{defn}
For a prime $p$ and $\delta\ge1$ put $R_{p,\delta}=\varphi_{p,\delta}(\mathbb{Z}_{\ge0})$.
\end{defn}

\begin{remark}
In view of \eqref{varphi21}, \eqref{varphip1}, and \eqref{varphipdelta}, it is clear that $|R_{2,1}|=2$ and $|R_{p,\delta}|=3$ otherwise. Also, any nonempty fiber of $\varphi_{p,\delta}$ is one of
\begin{equation*}
\{0\},\{1\},\{0,1\},\mathbb{Z}_{\ge2}=\{z\in\mathbb{Z}:z\ge2\}.
\end{equation*}
\end{remark}

Following directly from \eqref{varphi21}, \eqref{varphip1}, and \eqref{varphipdelta}, we have the following:

\begin{lemma}\label{firsttrans}
Let $p$ be a prime, $\delta\ge1$, $u\in R_{p,\delta}$, and $\mu\ge0$. Then we have:
\begin{enumerate}
\item In case $\varphi^{-1}_{p,\delta}(u)=\{0\}$, for $g\ge0$,
\begin{eqnarray}\label{firsttrans1}
\varphi_{p,\delta}(\mu+g)=u\iff \mu+g=0\iff \left\{ \begin{array}{ll}
g=0 & (\mu=0) \\
\text{impossible} & (\mu\ge1), \\
\end{array} \right.
\end{eqnarray}
\item In case $\varphi^{-1}_{p,\delta}(u)=\{1\}$, for $g\ge0$,
\begin{eqnarray}\label{firsttrans2}
\varphi_{p,\delta}(\mu+g)=u\iff \mu+g=1\iff \left\{ \begin{array}{ll}
g=1-\mu & (\mu=0,1) \\
\text{impossible} & (\mu\ge1), \\
\end{array} \right.
\end{eqnarray}
\item In case $\varphi^{-1}_{p,\delta}(u)=\{0,1\}$, for $g\ge0$,
\begin{eqnarray}\label{firsttrans3}
\varphi_{p,\delta}(\mu+g)=u\iff \mu+g\in\{0,1\}\iff \left\{ \begin{array}{ll}
g\le1 & (\mu=0) \\
g=0 & (\mu=1), \\
\text{impossible} & (\mu\ge2), \\
\end{array} \right.
\end{eqnarray}
\item In case $\varphi^{-1}_{p,\delta}(u)=\mathbb{Z}_{\ge2}$, for $g\ge0$,
\begin{eqnarray}\label{firsttrans4}
\varphi_{p,\delta}(\mu+g)=u\iff \mu+g\ge2\iff \left\{ \begin{array}{ll}
g\ge2-\mu & (\mu=0,1) \\
\text{always true} & (\mu\ge2). \\
\end{array} \right.
\end{eqnarray}
\end{enumerate}
Here impossible means that no $g\ge0$ can be found to fulfill $\varphi_{p,\delta}(\mu+g)=u$, and that always true means that all $g\ge0$ fulfills $\varphi_{p,\delta}(\mu+g)=u$.
\end{lemma}

\section{The case of $n=819$}\label{819}
\setcounter{equation}{0}
\setcounter{theorem}{0}

We consider the case $n=819$ of the main theorem Theorem \ref{main}. We have the prime factorizations
\begin{equation*}
\begin{aligned}
819=&3^2\cdot 7\cdot 13,\\
918=&2\cdot 3^3\cdot 17.
\end{aligned}
\end{equation*}
Let the prime factorization of $\rho_{c,3}$ be
\begin{equation*}
\rho_{c,3}=3^{g_1}\cdot7^{g_2}\cdot13^{g_3}\cdot17^{g_4}\cdot b,
\end{equation*}
where $(b,3\cdot7\cdot13\cdot17)=1$. The numbers $g_1,g_2,g_3,g_4,b$ obviously depend on $c$, but we have suppressed the notation for simplicity. Now
\begin{equation*}
\begin{aligned}
819\rho_{c,3}=&3^{2+g_1}\cdot 7^{1+g_2}\cdot 13^{1+g_3}\cdot17^{g_4}\cdot b,\\
r(819\rho_{c,3})=918\rho_{c,3}=&2\cdot3^{3+g_1}\cdot 7^{g_2}\cdot 13^{g_3}\cdot17^{1+g_4}\cdot b.
\end{aligned}
\end{equation*}
Applying the additive function $v$ to these equations
\begin{equation*}
\begin{aligned}
v(819\rho_{c,3})=&v(3^{2+g_1})+v(7^{1+g_2})+v(13^{1+g_3})+v(17^{g_4})+v(b),\\
v(r(819\rho_{c,3}))=v(918\rho_{c,3})=&v(2)+v(3^{3+g_1})+v(7^{g_2})+v(13^{g_3})+v(17^{1+g_4})+v(b).
\end{aligned}
\end{equation*}
Hence $819\rho_{c,3}$ is a $v$-palindromic number if and only if the above two quantities are equal, that is, after rearranging
\begin{equation*}
(v(7^{1+g_2})-v(7^{g_2}))+(v(13^{1+g_3})-v(13^{g_3}))=2+(v(3^{3+g_1})-v(3^{2+g_1}))+(v(17^{1+g_4})-v(17^{g_4})).
\end{equation*}
In terms of the functions $\varphi_{p,\delta}$ of Section \ref{functions}, this becomes
\begin{equation}\label{relationb}
\varphi_{7,1}(g_2)+\varphi_{13,1}(g_3)=2+\varphi_{3,1}(2+g_1)+\varphi_{17,1}(g_4).
\end{equation}
Since $2+g_1\ge2$, by \eqref{varphip1}, $\varphi_{3,1}(2+g_1)=1$, therefore \eqref{relationb} becomes
\begin{equation}\label{relation}
\varphi_{7,1}(g_2)+\varphi_{13,1}(g_3)=3+\varphi_{17,1}(g_4).
\end{equation}
Now consider the equation
\begin{equation}\label{relationu819}
u_2+u_3=3+u_4.
\end{equation}
We want to solve it for $u_2\in R_{7,1}$, $u_3\in R_{13,1}$, and $u_4\in R_{17,1}$. In view of \eqref{varphip1},
\begin{equation*}
R_{7,1}=\{7,2,1\},R_{13,1}=\{13,2,1\},R_{17,1}=\{17,2,1\}.
\end{equation*}
By trying all possibilities we see that the only solutions are $(u_2,u_3,u_4)=(7,13,17),(2,2,1)$. Whence \eqref{relation} is satisfied if and only if
\begin{equation*}
\begin{aligned}
(\varphi_{7,1}(g_2),\varphi_{13,1}(g_3),\varphi_{17,1}(g_4))=&(7,13,17)\text{ or}\\
(\varphi_{7,1}(g_2),\varphi_{13,1}(g_3),\varphi_{17,1}(g_4))=&(2,2,1).
\end{aligned}
\end{equation*}
We first consider when $(\varphi_{7,1}(g_2),\varphi_{13,1}(g_3),\varphi_{17,1}(g_4))=(7,13,17)$. By Lemmas \ref{firsttrans} (or more easily just by looking at \eqref{varphip1}), \ref{secondtrans}, \ref{divisiblelemma}, and Table,
\begin{equation}\label{condition1819}
\begin{aligned}
\varphi_{7,1}(g_2)=7\iff& g_2=0\iff 7\nmid\rho_{c,3}\iff h_{7,3}\nmid c\iff 2\nmid c\\
\varphi_{13,1}(g_3)=13\iff& g_3=0\iff 13\nmid\rho_{c,3}\iff h_{13,3}\nmid c\iff 2\nmid c\\
\varphi_{17,1}(g_4)=17\iff& g_4=0\iff 17\nmid\rho_{c,3}\iff h_{17,3}\nmid c\iff 16\nmid c.
\end{aligned}
\end{equation}
Hence $(\varphi_{7,1}(g_2),\varphi_{13,1}(g_3),\varphi_{17,1}(g_4))=(7,13,17)$ simply when $c$ is odd.
We next consider when $(\varphi_{7,1}(g_2),\varphi_{13,1}(g_3),\varphi_{17,1}(g_4))=(2,2,1)$. Similarly we have
\begin{equation}\label{condition2819}
\begin{aligned}
\varphi_{7,1}(g_2)=2\iff& g_2=1\iff 7\mid\rho_{c,3}\text{ and }7^2\nmid\rho_{c,3}\iff 2\mid c\text{ and }14\nmid c\\
\varphi_{13,1}(g_3)=2\iff& g_3=1\iff 13\mid\rho_{c,3}\text{ and }13^2\nmid\rho_{c,3}\iff 2\mid c\text{ and }26\nmid c\\
\varphi_{17,1}(g_4)=1\iff& g_4\ge2\iff 17^2\mid\rho_{c,3}\iff 272\mid c.
\end{aligned}
\end{equation}
Hence $(\varphi_{7,1}(g_2),\varphi_{13,1}(g_3),\varphi_{17,1}(g_4))=(2,2,1)$ precisely when $272\mid c$ and $(c,7\cdot13)=1$. Hence we have established that

\begin{theorem}\label{819t}
$819\rho_{c,3}$ is $v$-palindromic if and only if $c$ is odd or if $272\mid c$ and $(c,7\cdot13)=1$.
\end{theorem}

From the above theorem, we immediately see that $c(819)=1$ (refer to definitions in Definition \ref{defnofwc}). 

We see that $819\rho_{c,3}$ is $v$-palindromic if and only if all 3 conditions in \eqref{condition1} hold, or if all 3 conditions in \eqref{condition2} hold. Now these conditions have the same truth values when $c$ increases by $\operatorname{lcm}(16,14,26,272)=24752$. Hence $\omega=24752$ is a period of $819$. With some work, it can be shown that actually it is the smallest period, that is, $\omega(819)=24752$.

\section{Proof of the main theorem}\label{proof}
\setcounter{equation}{0}
\setcounter{theorem}{0}

We now enter the proof of the main theorem and this is essentialy writing the discussion about $819$ in the previous section in the general setting.

Let the prime factorizations of $n$ and $r(n)$ be
\begin{equation*}
\begin{aligned}
n=&p_1^{e_1}p^{e_2}_2\ldots p^{e_m}_m,\\
r(n)=&p_1^{f_1}p^{f_2}_2\ldots p^{f_m}_m,
\end{aligned}
\end{equation*}
where we have done the factorization over the set of primes which divide one of $n$ or $r(n)$, setting $e_i=0$ or $f_i=0$ if necessary. Since $n\ne r(n)$, $e_i\ne f_i$ for some $i$. Let the set of $i$ such that $e_i\ne f_i$
\begin{equation}\label{indices}
i_1<i_2<\ldots<i_t.
\end{equation}

Let the prime factorization of $\rho_{c,k}$ be
\begin{equation}\label{rhockfactorization}
\rho_{c,k}=p^{g_1}_1p^{g_2}_2\ldots p^{g_m}_mb,
\end{equation}
where $(b,p_1p_2\ldots p_m)=1$. The $g_1,g_2,\ldots,g_m,b$ obviously depends on $c$, but we suppress it from our notation for simplicity. Then
\begin{equation*}
\begin{aligned}
n\rho_{c,k}=&p^{e_1+g_1}_1p^{e_2+g_2}_2\ldots p^{e_m+g_m}_mb,\\
r(n\rho_{c,k})=r(n)\rho_{c,k}=&p^{f_1+g_1}_1p^{f_2+g_2}_2\ldots p^{f_m+g_m}_mb.
\end{aligned}
\end{equation*}
Taking their $v$, we have
\begin{equation*}
\begin{aligned}
v(n\rho_{c,k})=&\sum^m_{i=1}v(p^{e_i+g_i}_i)+v(b),\\
v(r(n\rho_{c,k}))=&\sum^m_{i=1}v(p^{f_i+g_i}_i)+v(b).
\end{aligned}
\end{equation*}
Hence $n\rho_{c,k}$ is $v$-palindromic, that is, $v(n\rho_{c,k})=v(r(n\rho_{c,k}))$, if and only if
\begin{equation}\label{rule1}
\sum^m_{i=1}(v(p^{e_i+g_i}_i)-v(p^{f_i+g_i}_i))=0.
\end{equation}
Whence $e_i=f_i$, of course the term $v(p^{e_i+g_i}_i)-v(p^{f_i+g_i}_i)=0$, so by \eqref{indices}, \eqref{rule1} is equivalent to
\begin{equation}\label{rule2}
\sum^t_{j=1}(v(p^{e_{i_j}+g_{i_j}}_{i_j})-v(p^{f_{i_j}+g_{i_j}}_{i_j}))=0.
\end{equation}
But this is a cumbersome notation, so we just write $p_{i_j}$ as $p_j$, $e_{i_j}$ as $e_j$, $f_{i_j}$ as $f_j$, and $g_{i_j}$ as $g_j$, which will not cause confusion from here on because we will not be referring to the other prime factors or exponents hereafter. Consequently \eqref{rule2} becomes
\begin{equation}\label{rule3}
\sum^t_{j=1}(v(p^{e_j+g_j}_j)-v(p^{f_j+g_j}_j))=0.
\end{equation}
We also write
\begin{equation*}
\begin{aligned}
\delta_j=&e_j-f_j,\\
\mu_j=&\min(e_j,f_j)\\
\alpha_j=&\mu_j+g_j,
\end{aligned}
\end{equation*}
for $1\le j\le t$. Then it is clear that the left-hand-side of \eqref{rule3} can be rewritten, using the functions $\varphi_{p,\delta}$ of Section \ref{functions}, as
\begin{equation}\label{rule4}
\sum^t_{j=1}(v(p^{e_j+g_j}_j)-v(p^{f_j+g_j}_j))=\sum^t_{j=1}\operatorname{sgn}(\delta_j)(v(p^{\alpha_j+|\delta_j|}_j)-v(p^{\alpha_j}_j))=\sum^t_{j=1}\operatorname{sgn}(\delta_j)\varphi_{p_j,|\delta_j|}(\alpha_j),
\end{equation}
where $\operatorname{sgn}$ is the sign function with $\operatorname{sgn}(\delta_j)=1$ if $\delta_j>0$ and $\operatorname{sgn}(\delta_j)=-1$ if $\delta_j<0$. Now consider the equation
\begin{equation}\label{relationu}
\sum^t_{j=1}\operatorname{sgn}(\delta_j)u_j=0.
\end{equation}
Supposedly we can solve it for
\begin{equation*}
(u_1,u_2,\ldots,u_t)\in R_{p_1,|\delta_1|}\times R_{p_2,|\delta_2|}\times\cdots R_{p_t,|\delta_t|}.
\end{equation*}
Let the set of all solutions be
\begin{equation*}
U=\{u=(u_1,\ldots,u_t)\}.
\end{equation*}
Then we see that
\begin{equation*}
\sum^t_{j=1}\operatorname{sgn}(\delta_j)\varphi_{p_j,|\delta_j|}(\alpha_j)=0
\end{equation*}
holds if and only if for some $u\in U$,
\begin{equation*}
\varphi_{p_j,|\delta_j|}(\alpha_j)=u_j\quad\forall 1\le j\le t.
\end{equation*}
Summarizing up to now, we have shown that

\begin{lemma}\label{proofmiddle}
$n\rho_{c,k}$ is $v$-palindromic if and only if for some $u\in U$, $\varphi_{p_j,|\delta_j|}(\alpha_j)=u_j$ for all $1\le j\le t$.
\end{lemma}

Now let us consider just the "atomic" condition $\varphi_{p_j,|\delta_j|}(\alpha_j)=\varphi_{p_j,|\delta_j|}(\mu_j+g_j)=u_j$. By Lemmas \ref{firsttrans}, \ref{secondtrans}, and \ref{divisiblelemma},
\begin{eqnarray}\label{condition1}
\varphi_{p_j,|\delta_j|}(\mu_j+g_j)=u_j\iff \left\{ \begin{array}{ll}
g_j=0, & (\text{if \eqref{firsttrans1} and $\mu_j=0$, or \eqref{firsttrans2} and $\mu_j=1$, or \eqref{firsttrans3} and $\mu_j=1$}) \\
g_j=1, & (\text{if \eqref{firsttrans2} and $\mu_j=0$}) \\
g_j\le1, & (\text{if \eqref{firsttrans3} and $\mu_j=0$}) \\
g_j\ge1, & (\text{if \eqref{firsttrans4} and $\mu_j=1$}) \\
g_j\ge2, & (\text{if \eqref{firsttrans4} and $\mu_j=0$}) \\
\text{impossible}, & (\text{otherwise}) \\
\text{always true}. & (\text{if \eqref{firsttrans4} and $\mu_j\ge2$}) \\
\end{array} \right.
\end{eqnarray}
As the last two cases, "impossible" and "always true", never change (as $c$ varies), we exclude them from our consideration. By Lemma \ref{firsttrans}, we can continue the equivalences in \eqref{condition1} respectively (here we do not write out the cases as in \eqref{condition1}), recalling that $g_j=\operatorname{ord}_{p_j}(\rho_{c,k})$
\begin{eqnarray}\label{condition2}
\varphi_{p_j,|\delta_j|}(\mu_j+g_j)=u_j\iff \left\{ \begin{array}{ll}
p_j\nmid\rho_{c,k}, \\
p_j\mid\rho_{c,k}\text{ and } p^2_j\nmid\rho_{c,k}, \\
p^2_j\nmid\rho_{c,k}, \\
p_j\mid\rho_{c,k}, \\
p^2_j\mid\rho_{c,k}. \\
\end{array} \right.
\end{eqnarray}
In case $p_j\ne 2,5$, we can use Lemma \ref{secondtrans} to \eqref{condition2} to obtain, respectively
\begin{eqnarray}
\varphi_{p_j,|\delta_j|}(\mu_j+g_j)=u_j\iff \left\{ \begin{array}{ll}
h_{p_j,k}\nmid c, \\
h_{p_j,k}\mid c\text{ and } h_{p^2_j,k}\nmid c, \\
h_{p^2_j,k}\nmid c, \\
h_{p_j,k}\mid c, \\
h_{p^2_j,k}\mid c. \\
\end{array} \right.
\end{eqnarray}
However, in case $p_j=2,5$, by the Remark \ref{however}, \eqref{condition2} becomes
\begin{eqnarray}
\varphi_{p_j,|\delta_j|}(\mu_j+g_j)=u_j\iff \left\{ \begin{array}{ll}
\text{always true}, \\
\text{impossible}, \\
\text{always true}, \\
\text{impossible}, \\
\text{impossible}. \\
\end{array} \right.
\end{eqnarray}
Since in general $a\mid b\iff a\mid(b+b')$ if $a\mid b'$ ($a,b,b'\ge1$ arbitrary integers, the $b$ not the one introduced in \eqref{rhockfactorization}), we see that the truth of $\varphi_{p_j,|\delta_j|}(\mu_j+g_j)=u_j$ does not change if we increase $c$ by
\begin{equation}\label{omega}
\omega=\operatorname{lcm}\{h_{p_j,k},h_{p^2_j,k}:p_j\ne2,5\}.
\end{equation}
In view of Lemma \ref{proofmiddle}, whether $n\rho_{c,k}$ is $v$-palindromic depends only on the truths of the individual $\varphi_{p_j,|\delta_j|}(\mu_j+g_j)=u_j$.
Hence this $\omega$ serves as a possible $\omega$ as required by the main theorem.

\section{Further Problems}

  In the proof of the main theorem, we found constructively a possible $\omega$ in \eqref{omega}, let us denote it by $\omega_f(n)$. However whether or not $\omega_f(n)$ is the smallest period, i.e. $\omega(n)$, is still unclear, although we know by Theorem \ref{period} that $\omega(n)\mid \omega_f(n)$. The following is a table of $\omega_f(n)$, $\omega(n)$, and $c(n)$, for $n\le 56$ with $n<r(n)$, computed using Mathematica \cite{Math}. We can assume without loss of generality that $n<r(n)$ because the pattern for $n$ and $r(n)$ are exactly the same, i.e.
  \begin{equation*}
\begin{aligned}
\omega_f(n)=&\omega_f(r(n)),\\
\omega(n)=&\omega(r(n)),\\
c(n)=&c(r(n)).
\end{aligned}
\end{equation*}

  \bigskip
  \begin{center}
    \def\arraystretch{2.2}
  \begin{tabular}{c|c|c|c|c|c|c|c|c|c}
  \rowcolor[rgb]{0.9, 0.9, 0.9}  
  $n$ & $12$ & $13$ & $14$ & $15$ & $16$ & $17$ & $18$ & $19$ & $23$ \\
  \hline \hline
  $\omega_f(n)$ & $21$ & $6045$ & $4305$ & $136$ & $1830$ & $337960$ & $9$ & $15561$ & $253$ \\
  \hline
  $\omega(n)$ & $1$ & $6045$ & $1$ & $1$ & $1$ & $337960$ & $1$ & $15561$ & $1$ \\
  \hline
  $c(n)$ & $\infty$ & $15$ & $\infty$ & $\infty$ & $\infty$ & $280$ & $1$ & $819$ & $\infty$ \\
  \hline \hline
  \rowcolor[rgb]{0.9, 0.9, 0.9}
  $n$ & $24$ & $25$ & $26$ & $27$ & $28$ & $29$ & $34$ & $35$ & $36$ \\
  \hline \hline
  $\omega_f(n)$ & $21$ & $39$ & $6045$ & $9$ & $4305$ & $102718$ & $122808$ & $14469$ & $21$ \\
  \hline
  $\omega(n)$ & $1$ & $1$ & $6045$ & $1$ & $1$ & $1$ & $1$ & $1$ & $1$ \\
  \hline
  $c(n)$ & $\infty$ & $\infty$ & $15$ & $\infty$ & $\infty$ & $\infty$ & $\infty$ & $\infty$ & $\infty$ \\
  \hline \hline
  \rowcolor[rgb]{0.9, 0.9, 0.9}
  $n$ & $37$ & $38$ & $39$ & $45$ & $46$ & $47$ & $48$ & $49$ & $56$ \\
  \hline \hline
  $\omega_f(n)$ & $32412$ & $581913$ & $6045$ & $9$ & $253$ & $119991$ & $21$ & $22701$ & $273$ \\
  \hline
  $\omega(n)$ & $32412$ & $1$ & $6045$ & $1$ & $1$ & $1$ & $21$ & $22701$ & $273$ \\
  \hline
  $c(n)$ & $12$ & $\infty$ & $15$ & $\infty$ & $\infty$ & $\infty$ & $3$ & $3243$ & $3$ \\
\end{tabular}
\end{center}
\bigskip
From this table, it seems that we always have $\omega(n)=1$ or $\omega(n)=\omega_f(n)$. Therefore we make the following conjecture

\begin{conjecture}
Let $n$ be a natural number with $n\in\mathbb{N}_{\ne10}$ and $n\ne r(n)$. Then $\omega(n)=1$ or $\omega(n)=\omega_f(n)$.
\end{conjecture}

For the third rows, i.e. the rows of values of $c(n)$, $\infty$ means that by concatenating the decimal digits of $n$ any number of times, no $v$-palindromic number will be reached; otherwise $c(n)$ is the least number of times one have to concatenate the decimal digits of $n$ to reach a $v$-palindromic number. Therefore we can consider such a problem raised by Michel Marcus

\begin{problem}\label{problem}
  Is there a simple way to determine whether $c(n)=\infty$ or not?
\end{problem}

Finally, it seems that for most $n$, $c(n)=\infty$. In fact, it can be shown that all the numbers in \eqref{2178} have $c(n)=\infty$, so in particular there are infinitely many such numbers. Hence it is natural to conjecture

\begin{conjecture}
  Let $S=\{n\in\mathbb{N}:10\nmid n, n<r(n)\}$ and let $T=\{n\in S:c(n)=\infty\}$. Then the asymptotic density of $T$ in $S$ is $1$.
\end{conjecture}

\section{Some Sequences}

After I released my manuscript, I had some correspondneces with Michel Marcus. Inspired by my manuscript, Michel Marcus created the entries A338038, A338039, A338166, and A338371 of the OEIS \cite{S}. A338038 is the function $v(n)$ and A338039 is the sequence of $v$-palindromic numbers. A338371 is the sequence of integers $n>0$ such that $10\nmid n$, $n\ne r(n)$, and $c(n)<\infty$. 

\section{Acknowledgements}

The author is grateful for the careful reading by Prof. Kohji Matsumoto and Prof. Hiroshi Suzuki. The author also want to thank Michel Marcus for valuable correspondences.

\end{document}